\documentclass[11pt]{article}

\usepackage{amsmath,amssymb,amsthm,setspace,graphicx}
\usepackage[text={6.5in,8.4in}, top=1.3in]{geometry}

\newtheorem{maintheorem}{Theorem}
\newtheorem{theorem}{Theorem}[section]
\newtheorem{proposition}[theorem]{Proposition}
\newtheorem{lemma}[theorem]{Lemma}

\begin{document}

\title{Induced subgraphs with many repeated degrees}

\author{
	Yair Caro 
	\thanks{Department of Mathematics, University of Haifa at Oranim, Tivon 3600600, Israel. Email: yacaro@kvgeva.org.il}
	\and
	Raphael Yuster
	\thanks{Department of Mathematics, University of Haifa, Haifa 3498838, Israel. Email: raphy@math.haifa.ac.il~.
	Research supported in part by the Israel Science Foundation (grant No. 1082/16).}
}

\date{}

\maketitle

\setcounter{page}{1}

\begin{abstract}
	Erd\H{o}s, Fajtlowicz and Staton asked for the least integer $f(k)$ such that every graph with more than $f(k)$ vertices has an induced regular subgraph with at least $k$ vertices. Here we consider the following relaxed notions.
	Let $g(k)$ be the least integer such that every graph with more than $g(k)$ vertices
	has an induced subgraph with at least $k$ repeated degrees and let $h(k)$ be the least integer such that every graph with more than $h(k)$ vertices has an induced subgraph with at least $k$ maximum degree vertices.
	We obtain polynomial lower bounds for $h(k)$ and $g(k)$ and nontrivial linear upper bounds when the host graph has bounded maximum degree.
\end{abstract}

\section{Introduction}

We consider undirected simple and finite graphs.
Erd\H{o}s, Fajtlowicz and Staton (c.f. \cite{CG-1998}, page 85) asked for the order of magnitude of the least integer $f(k)$ such that every graph with more than $f(k)$ vertices has an induced regular subgraph with at least $k$ vertices. Clearly $f(k) \le R(k)-1$ where $R(k)$ is the diagonal Ramsey number, but already small values of $k$ suggest that it is smaller. They conjectured that $f(k)$ is less than exponential in $k$ and this major open problem is still unresolved. The best lower bound for $f(k)$ is polynomial in $k$. Alon, Krivelevich and Sudakov \cite{AKS-2008} proved that $f(k) = \Omega(k^2/\sqrt{\log k})$ improving an earlier result of Bollobás (c.f. \cite{CG-1998}) who proved
$f(k) = \Omega(k^{2-\epsilon})$ for every $\epsilon > 0$.

In this paper we consider a version of the aforementioned question which relaxes the regularity requirement.
For a graph $G$, let $rep(G)$ be the maximum multiplicity of a mode (i.e. most common value) of the degree sequence of $G$. This parameter, also called the {\em repetition number} has been recently studied by several researchers,
see \cite{CSY-2014,CW-2009}.
Likewise, let $maxrep(G)$ be the number of vertices with maximum degree in $G$; this has recently been studied in
\cite{CY0-2010,GP-2018}. Let, therefore, $g(k)$ be the least integer such that every graph with more than $g(k)$ vertices has an induced subgraph with repetition number at least $k$ and let $h(k)$ be the least integer such that every graph with more than $h(k)$ vertices has an induced subgraph with at least $k$ maximum degree vertices.
We clearly have
$$
\Theta(k) \le g(k) \le h(k) \le f(k) \le R(k)-1
$$
and small values of $k$ show that all of these parameters are separated.

Our first main contribution in this paper is a polynomial lower bound for $g(k)$ and hence for $h(k)$.

\begin{maintheorem}\label{t:1}
There exists an absolute constant $c$ such that $g(k) \ge c k^{3/2}/(\log k)^{1/2}$.
\end{maintheorem}

Determining these functions even with the further restriction that the graph in question has bounded maximum degree also seems challenging. Let $g(k,d)$ denote the least integer such that every graph with maximum degree at most $d$
and more than $g(k,d)$ vertices has an induced subgraph $H$ with $rep(H) \ge k$. Similarly define $h(k,d)$ and $f(k,d)$.
Again, we have here that for every nonnegative integer $d$, $g(k,d) \le h(k,d) \le f(k,d) \le (k-1)(d+1)$ as in every
graph with $(k-1)(d+1)+1$ vertices and with maximum degree $d$ there is an independent set of size $k$.
Once again, it is not difficult to construct some small examples of distinct pairs $(k,d)$ 
showing that these $2$-valued functions are separated.

For very small fixed $d$, the values of $f(k,d)$ and $g(k,d)$ are easy to obtain.
For example, it clearly holds that $g(k,0)=h(k,0)=f(k,0)=k-1$.
It is also easy to show that $h(k,1)=g(k,1)=f(k,1) = \lfloor 1.5(k-1) \rfloor$.
For $d=2$ things are only slightly more involved but all parameters can be precisely determined as the host graphs in this case 
are just disjoint unions of cycles, paths, and isolated vertices. For example, it is not difficult to prove that
$h(k,2)=2(k-1)$ for all odd $k$ and $h(k,2)=2 k - 3$ for all even $k$ as was shown in \cite{CLZ-preprint}.
However, already for $d=3$, even the asymptotic behaviors of $h(k,3)$ as well as $g(k,3)$ and $f(k,3)$ seem elusive.
It is not difficult to show that for every fixed $d$, each of $h(k,d)/k$, $g(k,d)/k$, $f(k,d)/k$ has a limit.
The following theorem provides nontrivial upper and lower bounds for the limits of the first two, while Proposition \ref{p:1} provides an upper bound for the limit of the third.
\begin{maintheorem}\label{t:2}
There exists an absolute constant $c > 0$ such that 
for every $d \ge 2$,
$$
	c \left(\frac{d}{\log d}\right)^{1/3} \le \lim_{k \rightarrow \infty} \frac{g(k,d)}{k} \le  \lim_{k \rightarrow \infty} \frac{h(k,d)}{k} \le \frac{d}{2}+1\;.
$$
\end{maintheorem}
Regarding the first nontrivial case $d=3$ we obtain the following more specific bounds.
\begin{maintheorem}\label{t:3}
	$$
	\frac{53}{24} \le \lim_{k \rightarrow \infty} \frac{h(k,3)}{k} \le \frac{5}{2}\;,
	\qquad
	\frac{13}{6} \le \lim_{k \rightarrow \infty} \frac{g(k,3)}{k} \le \frac{12}{5}\;.
	$$
	
\end{maintheorem}
In the next section we prove the general lower bound, namely Theorem \ref{t:1}.
Section 3 considers bounded degree graphs where we prove Theorems \ref{t:2} and \ref{t:3}.
The final section contains some concluding remarks and open problems.

\section{A polynomial lower bound for repetition of induced subgraphs}

\begin{proof}[Proof of Theorem \ref{t:1}.]
Our construction is probabilistic and follows the non-uniform random graph model as in \cite{AKS-2008}.
Let $\overline{p}=(p_1,\ldots,p_n)$ where $p_i=(1-\epsilon)/\sqrt{2}+i(4+\sqrt{2})\epsilon/({7n})$ and let $G(n,\overline{p})$ be the probability space of graphs on vertex set $[n]$ where the pair $(i,j)$ is an edge 
with probability $p_i p_j$ independently of all other pairs. In this proof we will use $\epsilon=1-\sqrt{2}/4$ so that
$p_i = 1/4+i/(2n)$.

Let $C > 0$ be an absolute constant to be chosen later and let $k=3 C n^{2/3}(\ln n)^{1/3}$.
We will assume that $k$ is an integer multiple of $3$ as this does not affect the asymptotic claim and also assume that $n$ is sufficiently large to satisfy the claimed inequalities.
For every $k \le t \le n$ we will prove that the probability that $G \sim G(n,\overline{p})$ has an induced subgraph $H$ with $t$ vertices having $rep(H) \ge k$ is less than $1/n$.
Hence, by the union bound it will follow that with positive probability, a graph $G \sim G(n,\overline{p})$ has no subgraph with repetition number at least $k$ and therefore
$g(k) \ge n = c k^{3/2}/(\log k)^{1/2}$ for a suitable absolute constant $c$, hence Theorem \ref{t:1} follows.

Fix some $t$ with $k \le t \le n$ and fix some $T \subseteq [n]$ with $t=|T|$. Let $H=G[T]$ where $G \sim G(n,\overline{p})$. We will prove that the probability that $rep(H) \ge k$
is less than $(n \binom{n}{t})^{-1}$ and as there are only $\binom{n}{t}$ choices for $T$, this will prove that the probability of $G$ having an induced subgraph $H$ with $t$ vertices having $rep(H) \ge k$ is less than $1/n$,
as claimed.

Fix some $R \subseteq T$ with $|R|=k$. We will prove that the probability of all the vertices of $R$ having the same degree
in $H$ is less than $(\binom{t}{k})^{-1}(n \binom{n}{t})^{-1}$ hence it will imply the probability that $rep(H) \ge k$
is less than $(n \binom{n}{t})^{-1}$, as claimed.

Recall that the vertices of $R$ are a subset of the total order $[n]$ hence let $R=\{v_1,\ldots,v_k\}$ where $v_i < v_{i+1}$
for $i=1,\ldots,k-1$. Let $A$ denote the first $k/3$ vertices of $R$ and let $B$ denote the last $k/3$ vertices of $R$. Our goal is to upper bound the probability that the sum of the degrees in $H$ of the vertices in $A$
equals the sum of the degrees in $H$ of the vertices in $B$. This probability clearly upper bounds the probability that
all the vertices of $R$ have the same degree in $H$.

For a pair of vertices $(i,j)$, let $X_{i,j}$ denote the
indicator random variable which equals $1$ if $(i,j)$ is an edge.
Let $X_A$ denote the sum of the degrees of $A$ in $H$.
Let $Y_A$ denote the sum of the indicator random variables corresponding to pairs $(i,j)$ such that $i \in A$ and
$j \in T \setminus (A \cup B)$ and let $Z_A$ denote the sum of the indicator random variables corresponding to pairs $(i,j)$ such that both $i,j$ are in $A$. Analogously define $X_B$, $Y_B$, $Z_B$.
Notice that $X_B-X_A=Y_B+2 Z_B-Y_A-2 Z_A$ (indicator random variables corresponding to pairs with one endpoint in $A$ and another in $B$ are canceled out in the difference $X_B-X_A$).
Also notice that all the $2\binom{k/3}{2}+2(k/3)(t-2 k/3)$ indicator variables involved in forming $Y_A,Y_B,Z_A,Z_B$ are independent.

We next estimate $E[X_B-X_A]$.
First observe that $E[Z_B - Z_A] \ge 0$ since a pair of vertices in $B$ have a higher chance to be an edge than a pair of vertices in $A$. Therefore,
$$
E[X_B-X_A]=E[Y_B-Y_A] + 2 E[Z_B-Z_A] \ge E[Y_B-Y_A]\;.
$$
For every vertex $j \in T \setminus (A \cup B)$ and for every $i \in A$ and $i' \in B$ the
probability of
the pair $i,j$ to be an edge is less than the probability of the pair $(i',j)$ to be an edge by
at least $k/(24 n)$ since
$$
\left(\frac{1}{4}+\frac{i'}{2 n}\right) \left(\frac{1}{4}+\frac{j}{2 n}\right) - 
\left(\frac{1}{4}+\frac{i}{2 n}\right) \left(\frac{1}{4}+\frac{j}{2 n}\right) \ge
\left(\frac{i'-i}{8 n}\right) > \frac{k}{24 n}\;.
$$
It therefore follows that $E[Y_B-Y_A] \ge k(t-2 k)\frac{k}{24 n}$.
Consequently, since $t \ge 3 k$ we obtain
$$
E[X_B-X_A] \ge E[Y_B-Y_A] \ge k(t-2 k)\frac{k}{24 n} \ge \frac{tk^2}{72 n}\;.
$$

Let us next consider the random variable $S=\frac{1}{2}E[X_B-X_A]-\frac{1}{2}(X_B-X_A)$.
Then $E[S] = 0$ and $S$ is the sum of $q=2\binom{k/3}{2}+2(k/3)(t-2 k/3)$ independent random variables
where each of these random variables takes only two values and those two values are at most $1$ apart.
Then by a large deviation result of Chernoff (see \cite{AS-2004}, Theorem A.1.18),
$$
\Pr[S > a] < e^{-\frac{2a^2}{q}}\;.
$$
We therefore obtain:
\begin{align*}
	\Pr[X_B = X_A] & = \Pr\left[S = \frac{E[X_B-X_A]}{2}\right] \\
	& \le \Pr\left[S > \frac{E[X_B-X_A]}{3}\right]\\
	& \le \exp\left(-\frac{2\left(\frac{E[X_B-X_A]}{3}\right)^2}{q}\right)\\
	& \le \exp\left(-\frac{2 t^2 k^4}{9 \cdot 72^2 n^2 q}\right)\\
	& \le \exp\left(-\frac{t k^3}{9 \cdot 72^2 n^2}\right)\;.
\end{align*}
Recalling that $k=3 C n^{2/3}(\ln n)^{1/3}$ we obtain from the last inequality that for a sufficiently large absolute constant $C$,
\begin{align*}
\Pr[X_B = X_A] & \le \exp\left(-\frac{tk^3}{9 \cdot 72^2 n^2}\right)\\
& \le \exp\left(-2 t \ln n\right)\\
& = \frac{1}{n^{2 t}}\\
& < \frac{1}{n\binom{t}{k}\binom{n}{t}}
\end{align*}
as required.
\end{proof}

\section{Repetition in induced subgraphs of bounded degree graphs}

We first observe that for every fixed $d$, the sequences $g(k,d)/k$ and $h(k,d)/k$ have a limit.
First observe that as mentioned in the introduction, both are bounded from above by $d+1$.
We next show that $g(2 k,d) \ge 2g(k,d)$ and $h(2k,d) \ge 2h(k,d)$ implying the claimed limit in
both cases. Consider a graph $G$ with maximum degree at most $d$ and with $g(k,d)$ vertices
for which every induced subgraph $H$ has $rep(H) < k$.
Take two vertex-disjoint copies of $G$, thereby obtaining a graph with
maximum degree at most $d$ and with $2g(k,d)$ vertices for which every induced subgraph $H$ has $rep(H) < 2k$. Thus, by its definition, $g(2k,d) \ge 2g(k,d)$. An identical argument holds for $h(k,d)$.

The proofs of Theorems \ref{t:2} and \ref{t:3} follow directly from the proofs of the following Lemmas
and proposition.

\begin{proposition}\label{p:0}
	There exists an absolute constant $c > 0$ such that for all $d \ge 2$,
	$$
	c \left(\frac{d}{\log d}\right)^{1/3} \le \lim_{k \rightarrow \infty} \frac{g(k,d)}{k}\;.
	$$
\end{proposition}
\begin{proof}
Let $d \ge 2$.
By Theorem \ref{t:1}, there exists an absolute constant $C > 0$ such that there are graphs with at most $d$ vertices and for which every induced subgraph has repetition less than $C d^{2/3} (\log d)^{1/3}$.
Fixing $d$, let $G_d$ be such a graph. Now take $k/(C d^{2/3} (\log d)^{1/3})$ pairwise vertex-disjoint
copies of $G_d$. The resulting graph, which clearly has maximum degree less than $d$,
has $k d^{1/3}/(C(\log d)^{1/3})$ vertices and every induced subgraph
has repetition less than $k$. It follows that $g(k,d)/k \ge d^{1/3}/(C (\log d)^{1/3})$
and the proposition follows.
\end{proof}

\begin{lemma}\label{l:hk3lower}
	For infinitely many $k$ it holds that $h(k,3) \ge (53/24)(k-1)$.
\end{lemma}
\begin{proof}
	Consider the following set of graphs $\{K_1, H_1,H_2,H_3\}$ where $K_1$ is the isolated vertex and
	$H_1,H_2,H_3$ are the graphs shown in Figure \ref{f:hk3}. Observe that each of them has maximum degree
	at most $3$. Now, suppose $k-1$ is a multiple of $24$. Construct a graph $G_k$ as follows.
	Take $(7/24)(k-1)$ disjoint copies of $K_1$, $(1/24)(k-1)$ disjoint copies of $H_1$,
	$(1/24)(k-1)$ disjoint copies of $H_2$ and $(1/6)(k-1)$ disjoint copies of $H_3$.
	The graph $G_k$ has
	$$
	1 \cdot \frac{7}{24}(k-1) + 6 \cdot \frac{1}{24}(k-1) + 8 \cdot \frac{1}{24}(k-1) + 8 \cdot \frac{1}{6}(k-1) = \frac{53}{24}(k-1)
	$$
	vertices. It is not hard to check that for an induced subgraph $H$ of $G_k$ we have $maxrep(H) \le k-1$
	as follows. A subgraph with maximum degree $0$ can consist of all copies of $K_1$, two vertices from a copy of $H_1$, three vertices from a copy of $H_2$ and three vertices from a copy of $H_3$, yielding a total of $k-1$. A subgraph with maximum degree $1$ can consist of $4$ vertices from each copy of $H_1,H_2,H_3$, yielding a total of $k-1$. A subgraph with maximum degree $2$ can consist of $3,5,4$ vertices from each copy of $H_1,H_2,H_3$ respectively, yielding a total of $k-1$.
	A subgraph with maximum degree $3$ can consist of $2,2,5$ vertices from each copy of
	$H_1,H_2,H_3$ respectively, yielding a total of $k-1$.
\end{proof}
\begin{figure}[h!]
	\begin{center}
	\includegraphics[scale=0.5, trim=0 380 200 40]{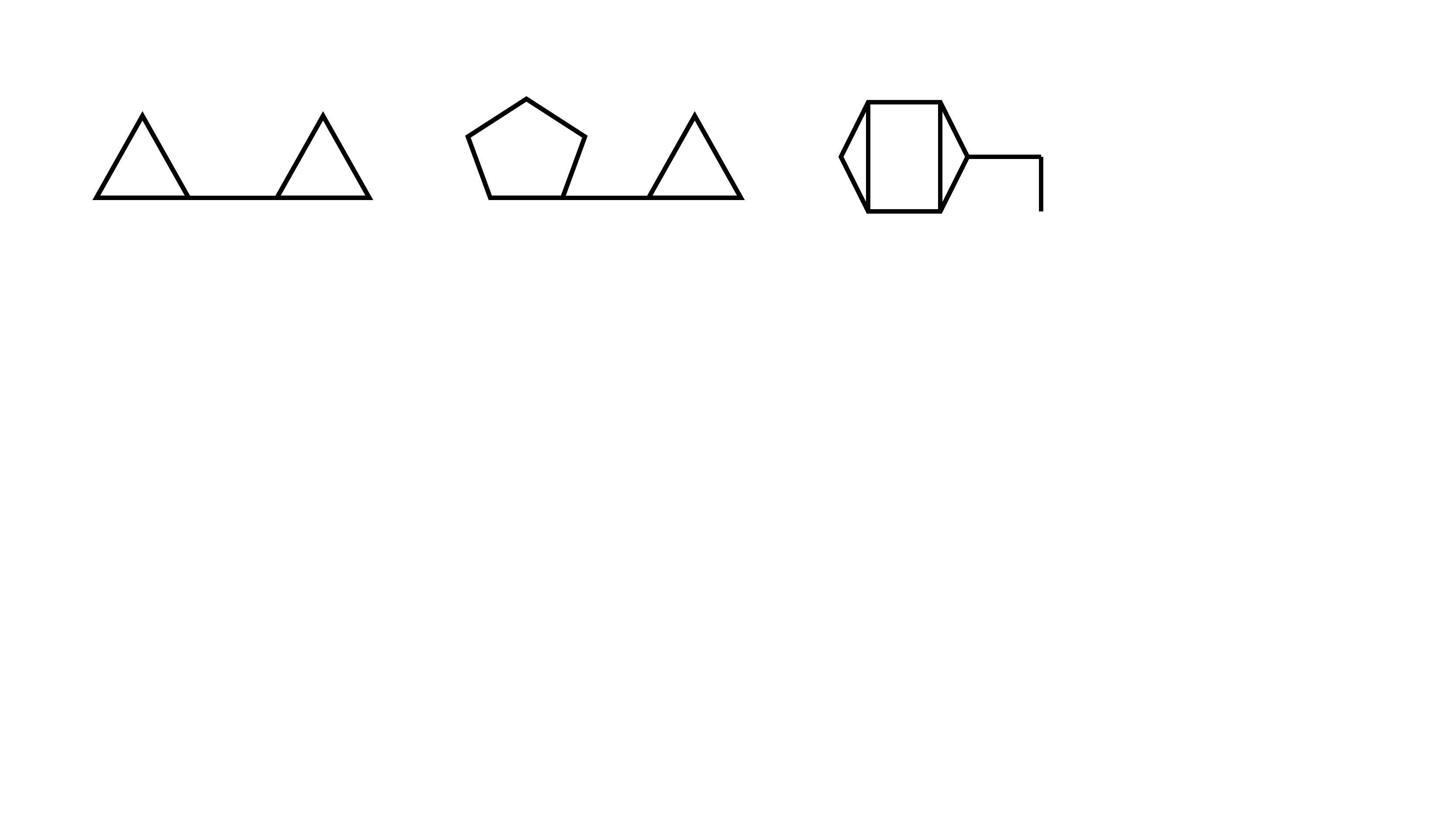}  
	\caption{The graphs $H_1,H_2,H_3$ of Lemma \ref{l:hk3lower} shown from left to right.}
	\label{f:hk3}
	\end{center}
\end{figure}
The reasoning behind the construction of Lemma \ref{l:hk3lower} is as follows.
A close to optimal construction for $h(k,3)$ is a disjoint union of bounded size components
since we have observed in the beginning of this section that a lower bound for $h(2k,3)$ is obtained
by taking two copies of an optimal construction for $h(k,3)$  and that the sequence divided by $k$
converges. So, given a positive integer constant $r$, we would like to optimize
over all graphs $G$ consisting of components of size at most $r$ each and maximum degree at most $3$. Denote by $D(r,3)$ the set of all connected graphs with at most $r$ vertices and maximum degree
at most $3$. If we can generate $D(r,3)$, we can use linear programming to decide upon the optimal
ratio of components of each possible element belonging to $D(r,3)$.
Construct
a linear programming instance as follows. For each $H \in D(r,3)$ define a nonnegative variable $x_H$.
Now maximize the sum $|H|x_H$ under the following constraints, one constraint for each $p=0,1,2,3$.
Let $c(H,p)$ denote the maximum number of vertices of maximum degree $p$ in an induced subgraph of $H$.
Then we require that the sum $c(H,p)x_H$ is at most $k-1$ (or one can normalize dividing
by $k-1$).
Now, if $r$ is too large, we do not have enough computing power to determine $D(r,3)$.
We have, however, determined $D(10,3)$, and have run the aforementioned linear program,
the result of which is the construction of Lemma \ref{l:hk3lower} which yielded, after normalization
$x_{K_1} = 7/24$, $x_{H_1}=1/24$, $x_{H_2}=1/24$  $x_{H_3}=1/6$.
The remaining variable are zero. In particular, this means that in any improved construction, if exists, one would
need to use components of size at least $11$.

Observe that the construction of Lemma \ref{l:hk3lower} fails as a lower bound construction for $g(k,3)$ since there is an induced subgraph (of maximum degree $3$) where the degree $2$ appears more than $k-1$ times.
\begin{lemma}\label{l:gk3lower}
	For infinitely many $k$ it holds that $g(k,3) \ge (13/6)(k-1)$.
\end{lemma}
\begin{proof}
	Consider the following set of graphs $\{K_1, K_2, P_4, Q\}$ where $P_4$ is the
	path on four vertices and $Q$ is the complement of $C_6$.
	Observe that each of them has maximum degree
	at most $3$. Now, suppose $k-1$ is a multiple of $6$. Construct a graph $G_k$ as follows.
	Take $(k-1)/6$ disjoint copies of each of $K_1$, $K_2$, $P_4$, $Q$.
	The graph $G_k$ has
	$$
	1 \cdot \frac{1}{6}(k-1) + 2 \cdot \frac{1}{6}(k-1) + 4 \cdot \frac{1}{6}(k-1) + 6 \cdot \frac{1}{6}(k-1) = \frac{13}{6}(k-1)
	$$
	vertices. It is not hard to check that for an induced subgraph $H$ of $G_k$ we have
	$rep(H) \le k-1$
	as follows. The degree $0$ can be repeated at most $1,1,2,2$ times in an induced subgraph of
	$K_1$, $K_2$, $P_4$, $Q$ respectively.
	The degree $1$ can be repeated at most $0,2,2,2$ times in an induced subgraph of
	$K_1$, $K_2$, $P_4$, $Q$ respectively.
	The degree $2$ can be repeated at most $0,0,2,4$ times in an induced subgraph of
	$K_1$, $K_2$, $P_4$, $Q$ respectively.
	The degree $3$ can be repeated at most $0,0,0,6$ times in an induced subgraph of
	$K_1$, $K_2$, $P_4$, $Q$ respectively.
	In all four cases, the yielded total is $k-1$.
\end{proof}

We now turn to the upper bounds. Our next lemma implies, in particular,
$\lim_{k \rightarrow \infty} \frac{h(k,3)}{k} \le \frac{5}{2}$.
\begin{lemma}\label{l:upper-1}
$h(k,d) \le (k-1)(d/2+1)$.
\end{lemma}
\begin{proof}
We prove the lemma by induction on $d$ where the base case $d=0$ is trivial.
Suppose $G$ is a graph with $n > (k-1)(d/2+1)$ vertices and with maximum degree at most $d$.
We must show that there is an induced subgraph $H$ with $maxrep(H) \ge k$.
If $G$ has $k$ or more vertices of degree $d$, we are done as we can take $H=G$.
So, assume otherwise. This means that the sum of the degrees of the vertices of $G$ is at most
$(d-1)n+k-1$ so $G$ has at most $n(d-1)/2+(k-1)/2$ edges.

We perform the following process which has $d-1$ stages.
At the beginning of stage $t$, we have an induced subgraph of $G$ denoted by $H_{d+1-t}$ with
maximum degree at most $d$. For $t=1$ set $H_d=G$. 
Stage $t$ consists of repeatedly removing vertices of degree $d+1-t$ form $H_{d+1-t}$
until we remain with a subgraph with maximum degree at most $d-t$ denoted by $H_{d-t}$.
Let $x_{d+1-t}$ be such that the number of vertices removed in stage $t$ is $x_{d+1-t}(k-1)$
and observe that we may assume that $0 \le x_i \le 1$ for $i=2,\ldots,d$ as otherwise we are
trivially done. When the process ends after stage $d-1$ we remain with a graph $H_1$ of maximum degree at most $1$.

The number of vertices of $H_1$ is
$$
n-(k-1)\sum_{i=2}^d x_i\;.
$$
The number of edges of $H_1$ is
$$
|E(G)| - (k-1)\sum_{i=2}^d ix_i \le n(d-1)/2+(k-1)/2 - (k-1)\sum_{i=2}^d ix_i\;.
$$
Now suppose that for some $2 \le i \le d$ we have $\sum_{j=i}^d x_j < n/(k-1)-(i+1)/2$. In this case
we have that the number of vertices of $H_{i-1}$ is
$$
n- (k-1)\sum_{j=i}^d x_j > (k-1)\left( \frac{i-1}{2}+1\right)
$$
which implies by the induction hypothesis that $H_{i-1}$ has
an induced subgraph $H$ with $maxrep(H) \ge k$.
So, we may assume that for all $2 \le i \le d$ it holds that $\sum_{j=i}^d x_j \ge n/(k-1)-(i+1)/2$.
In particular, this means that
\begin{equation}\label{e:sum}
\sum_{i=2}^d (i-1)x_i \ge \frac{n(d-1)}{k-1}-(d+2)(d+1)/4+1.5.
\end{equation}
Returning now to $H_1$, it has an independent set of size at least its number of vertices minus its number of edges, namely of size at least
\begin{align}\label{e:1}
& \quad \left(n-(k-1)\sum_{i=2}^d x_i\right) - \left(n(d-1)/2+(k-1)/2 - (k-1)\sum_{i=2}^d ix_i \right) \\
& \ge \; n\left(1-\frac{d-1}{2}\right)-\frac{k-1}{2}+n(d-1)+1.5(k-1)-(k-1)\frac{(d+2)(d+1)}{4} \nonumber\\
& = \; n\left(1+\frac{d-1}{2}\right)+(k-1)\left(1-\frac{(d+2)(d+1)}{4} \right) \nonumber\\
& > \;(k-1)\left[\left(\frac{d}{2}+1\right)\left(1+\frac{d-1}{2}\right)+1 - \frac{(d+2)(d+1)}{4} \right] \nonumber\\
& = \; k-1\;. \nonumber
\end{align}
So, there is an independent set of size at least $k$ in $G$ which is, in particular 
an induced subgraph $H$ with $maxrep(H) \ge k$.
\end{proof}

We can slightly modify the proof of Lemma \ref{l:upper-1} to obtain an upper bound for $g(k,d)$.
This upper bound gives, in particular, that $\lim_{k \rightarrow \infty} \frac{g(k,3)}{k} \le \frac{12}{5}$.
\begin{lemma}\label{l:upper-2}
	$g(k,d) \le (k-1)(2d+6)/5$ for $d \ge 2$ and $g(k,1) \le 1.5(k-1)$.
\end{lemma}
\begin{proof}
	We prove the lemma by induction on $d$ where the base case $d=1$ is an easy exercise.
	Assume now that $d \ge 2$ and suppose $G$ is a graph with
	$n > (k-1)(2d+6)/5$ vertices and with maximum degree at most $d$.
	We must show that there is an induced subgraph $H$ with $rep(H) \ge k$.
	If $G$ has $k$ or more vertices of the same degree we are done as we can take $H=G$.
	So, assume otherwise. This means that there are at most $k-1$ vertices of any given degree.
	Let $r \ge 2$ be any integer such that $r(k-1) \le n$ then the 
	sum of the degrees of the vertices of $G$ is at most
	$$
	d n - (k-1)-2(k-2)-\cdots-(r-1)(k-1)-r(n-r(k-1)) = (d-r)n -(k-1)\binom{r}{2}+r^2(k-1)\;.
	$$
	So the number of edges of $G$ is at most
	$$
	\frac{(d-r)n}{2}+(k-1)\left(\frac{r^2}{2}-\frac{r(r-1)}{4}\right)\;.
	$$
	We perform the exact same process  with $d-1$ stages as in Lemma \ref{l:upper-1}
	and using the same notations of $H_i$ and $x_i$ as in that lemma.
	When the process ends after stage $d-1$ we remain with a graph $H_1$ of maximum degree at most $1$.
	The number of vertices of $H_1$ is
	$n-(k-1)\sum_{i=2}^d x_i$ and the number of edges of $H_1$ is
	$$
	|E(G)| - (k-1)\sum_{i=2}^d ix_i \le  \frac{(d-r)n}{2}+(k-1)\left(\frac{r^2}{2}-\frac{r(r-1)}{4}\right) - (k-1)\sum_{i=2}^d ix_i\;.
	$$
	Now suppose that for some $3 \le i \le d$ we have $\sum_{j=i}^d x_j < n/(k-1)-(2i+4)/5$. In this case
	we have that the number of vertices of $H_{i-1}$ is
	$$
	n- (k-1)\sum_{j=i}^d x_j > (k-1)\left( \frac{2(i-1)+6}{5}\right)
	$$
	which implies by the induction hypothesis that $H_{i-1}$ has
	an induced subgraph $H$ with $rep(H) \ge k$.
	So, we may assume that for all $3 \le i \le d$ it holds that $\sum_{j=i}^d x_j \ge n/(k-1)-(2i+4)/5$.
	In the case $i=2$, if we have $\sum_{j=2}^d x_j < n/(k-1)-1.5$ then 
	the number of vertices of $H_1$ is
	$$
	n-(k-1) \sum_{j=2}^d x_j > 1.5(k-1)
	$$
	which implies that $H_1$ has an induced subgraph $H$ with $rep(H) \ge k$.
	So, we may assume that $\sum_{j=2}^d x_j \ge n/(k-1)-1.5$.
	In particular, this means that
	$$
	\sum_{i=2}^d (i-1)x_i \ge \frac{n(d-1)}{k-1}-\frac{4}{5}(d-1)-\frac{d(d+1)}{5}+0.5.
	$$
	Returning now to $H_1$, it has an independent set of size at least its number of vertices minus its number of edges, namely of size at least
	\begin{align*}
	& \quad \left(n-(k-1)\sum_{i=2}^d x_i\right) - \left( \frac{(d-r)n}{2}+(k-1)\left(\frac{r^2}{2}-\frac{r(r-1)}{4}\right) - (k-1)\sum_{i=2}^d ix_i \right)\\
	& \ge \; n\left(\frac{d+r}{2}\right)+(k-1)\left(\frac{r(r-1)}{4}-\frac{r^2}{2}+\frac{1}{2}-\frac{4}{5}(d-1)-\frac{d(d+1)}{5}\right)\\
	& > \; (k-1) \left[\left(\frac{2d+6}{5}\right) \left(\frac{d+r}{2}\right)+\frac{r(r-1)}{4}-\frac{r^2}{2}+\frac{1}{2}-\frac{4}{5}(d-1)-\frac{d(d+1)}{5}\right]\;.
	\end{align*}
	Now, the last inequality is at least $k-1$ if $r \ge 2$ and equals $k-1$ if $r=2$.
	So, there is an independent set of size at least $k$ in $G$ which is, in particular 
	an induced subgraph $H$ with $rep(H) \ge k$.
\end{proof}

\section{Concluding remarks and open problems}

Recall the important notion of quasi-random graphs defined by Chung, Graham, and Wilson \cite{CGW-1989}.
We say that a graph $G$ is quasi-random if for any subset of vertices $U \subseteq V(G)$ we have $e(U)=\frac{1}{2}\binom{|U|}{2} \pm o(n^2)$. In other words, it resembles the behavior of the random
graph $G(n,\frac{1}{2})$ with respect to edge distributions. It is well-known that quasi-random graphs possess many (though not all) properties of a typical random graph. Observe that in our proof of Theorem
\ref{t:1} we have used $\epsilon =1-\sqrt{2}/4$ but the proof works equally well with
any fixed $\epsilon > 0$ (this only affects the value of the constant $c$) or even when
$\epsilon$ is a function of $n$ tending slowly to zero (in which case we will still have
$g(k) \ge ck^{3/2}/\log k$). But then observe that the distribution $G(n,\overline{p})$ in the proof is a quasi-random graph with very high probability, since each pair is an edge with probability $\frac{1}{2} \pm o_n(1)$.
Consequently, this shows that there are quasi-random graphs with $n$ vertices for which every induced subgraph has repetition at most $n^{2/3} \log n$.
This should be compared to the result proved by Krivelevich, Sudakov, and Wormald \cite{KSW-2011} showing that with high probability a largest induced {\em regular} subgraph of $G(n,\frac{1}{2})$ has
about $n^{2/3}$ vertices. So, a random graph usually has a regular subgraph with about $n^{2/3}$ vertices
whereas already quasi-random graphs may not have even a repetition of about $n^{2/3}$ in an induced subgraph. Noticing this, it seems interesting to determine the typical maximum repetition of an induced subgraph of a random graph.

\vspace{5pt}
The functions $f(k,d)$, $g(k,d)$, $h(k,d)$ have the following sub-additive property.
\begin{lemma}\label{l:lovasz}
	Let $z \in \{f,g,h\}$, then $z(k,d_1)+z(k,d_2) \ge z(k,d_1+d_2+1)$.
\end{lemma}
\begin{proof}
	By a result of Lovász \cite{lovasz-1966}, the vertices of a graph $G$ with maximum degree $d_1+d_2+1$ can be partitioned into two parts $A_1$ and $A_2$ where the maximum degree of $G[A_i]$ is at most $d_i$
	for $i=1,2$. Now, if $G$ has more than $h(k,d_1)+h(k,d_2)$ vertices, there is some $i$ for which
	$G[A_i]$ has more than $h(k,d_i)$ vertices, so by definition $G[A_i]$ has an induced subgraph $H$ with
	$maxrep(H) \ge k$. The same argument holds for $g$ and $f$.
\end{proof}
While Lemma \ref{l:lovasz} is not strong enough to improve upon Lemmas \ref{l:upper-1} and \ref{l:upper-2},
it is useful for obtaining the following upper bound:
\begin{proposition}\label{p:1}
	$$ \lim_{k \rightarrow \infty} \frac{f(k,d)}{k} \le \frac{11d}{15}\;.$$
\end{proposition}
\begin{proof}
	We prove that $f(k,2) \le 11(k-1)/5$.
	Let $G$ be a graph with $n > 11(k-1)/5$ vertices and of maximum degree at most $2$. Let $t$
	denote the number of vertices on all the cycles of $G$. If $t \ge k$ we are done.
	Otherwise, let $G'$ denote the induced subgraph on the path components which consists of $n-t$ vertices. Let $t_0$ denote the size of an independent set of $G'$ and let $t_1$ denote the
	number of vertices in a maximum induced matching of $G'$. It is straightforward to verify that
	$t_0+t_1 \ge n-t$. Now, in each cycle of length $r$, the number of independent vertices plus the number of vertices in an induced matching is at least $4r/5$ where the worst case occurs for $C_5$.
	So, if $s_0$ denotes the size of an independent set of $G$ and $s_1$ denotes the
	number of vertices in a maximum induced matching of $G$ we have that $s_0+s_1 \ge n-t+4t/5$.
	So there is an induced regular subgraph of $G$ of order at least $n/2-t/10$.
	But
	$$
	\frac{n}{2}-\frac{t}{10} > \frac{11(k-1)}{10} - \frac{k-1}{10}=k-1\;.
	$$
	It is worth noting that the bound $11(k-1)/5$ is tight as for any $k$ such that $k-1$ is a multiple of $10$ we can take $(k-1)/5$ copies of $C_5$ and $3(k-1)/10$ copies of $P_4$ and obtain a graph
	with $11(k-1)/5$ vertices and no induced regular subgraph on $k$ vertices.
	Now, using Lemma \ref{l:lovasz} we obtain that $f(k,3 \cdot 2^r-1) \le (11/5)(k-1)2^r$
	and the upper bound on the limit follows.
\end{proof}

Finally, it would be interesting to obtain a sub-linear upper bound for the limit of $h(k,d)/k$ or improve
the polynomial lower bound for $g(k)$ which would imply, by Proposition \ref{p:0}, an improved lower bound for the limit of $g(k,d)/k$.

\bibliographystyle{plain}

\bibliography{references}

\end{document}